\documentclass[12pt]{amsart}
\usepackage[T1]{fontenc}
\usepackage{amssymb}
\newcommand{\cf}{\operatorname{cf}}

\newcommand{\dom}{\operatorname{dom}}

\newcommand{\su}{\operatorname{c}}

\newcommand{\sat}{\operatorname{sat}}

\newcommand{\invl}{\varprojlim}

\newcommand{\Raa }{\mathcal R}
\newcommand{\Bee }{\mathcal B}
\newcommand{\Cee }{\mathcal C}

\newcommand{\Pee }{\mathcal P}

\newcommand{\Tee }{\mathcal T}

\newcommand{\w}{\operatorname{w}}

\newcommand{\cl}{\operatorname{cl}}
\renewcommand{\int}{\operatorname{Int}}

\newtheorem{theorem}{Theorem}
\newtheorem{corollary}[theorem]{Corollary}

\newtheorem{lemma}[theorem]{Lemma}

\newtheorem{proposition}[theorem]{Proposition}

\author{Andrzej Kucharski}
\address{Andrzej Kucharski \\
 Institute of Mathematics, University of
Silesia \\
 ul. Bankowa 14, 40-007 Katowice}
\email{akuchar@math.us.edu.pl}

\begin{document}

\title{On open-open games of uncountable length.} 

\begin{abstract}
The aim of this note is to investigate  the open-open game of uncountable length. We introduce a cardinal number $\mu(X)$, which says how long the Player I has to play to ensure a victory. It is  proved that $\su(X)\leq\mu(X)\leq\su(X)^+$. We also introduce the class $\mathcal C_\kappa$ of topological spaces that can be represented as the inverse limit of $\kappa$-complete system $\{X_\sigma,\pi^\sigma_\rho,\Sigma\}$ with $\w(X_\sigma)\leq\kappa$ and skeletal bonding maps. It is shown that product of spaces which belong
to $\mathcal C_\kappa$ also belongs to this class and $\mu(X)\leq\kappa$ whenever $X\in\mathcal C_\kappa$ . 
\end{abstract}

\maketitle

\section{Introduction}
The following game is due to P. Daniels, K. Kunen  and H. Zhou \cite{dkz}: two
 players  take turns playing on a topological space $X$; a round consists of Player I choosing a nonempty open set  $U\subseteq X$; and Player II choosing a nonempty open set $V\subseteq U$; a round is played for each natural number. Player I wins the game if the union of  open sets which have been chosen by Player II is dense in $X$. This game is called the \textit{open-open game}.
 
In this note, we consider what happens if one drops restrictions on the length of  games.   If $\kappa$ is an infinite cardinal and 
   rounds are played for every ordinal number less then $\kappa,$ then this modification is called  \textit{the open-open game of length $\kappa$}. The examination of such games is  a continuation of \cite{kp7},  \cite{kp8} and \cite{kp9}.     
A cardinal number $\mu(X)$ is introduced such that $\su(X)\leq \mu(X)\leq\su(X)^+$. Topological spaces, which can be represented as an inverse limit of  $\kappa$-complete system $  \{ X_\sigma , 
\pi^\sigma_\varrho, \Sigma\}$ with $\w(X_\sigma )\leq\kappa$ and each $X_\sigma$ is $T_0$ space and skeletal
bonding map $\pi^\sigma_\varrho$, are listed as the class $\mathcal C_\kappa.$ If $\mu(X)=\omega,$ then $X\in C_\omega$. There exists a space $X$ with
$X\not\in C_{\mu(X)}$.  The class $\mathcal C_\kappa$ is closed under any Cartesian product. In particular, the cellularity number of $X^I$ is equal  $\kappa$ whenever $X\in\mathcal C_\kappa$.
This implies Theorem of D. Kurepa that  $\su(X^I)\leq 2^\kappa,$ whenever $\su(X)\leq\kappa$. 
  Undefined notions and symbols  are used in accordance with  books \cite{ch}, \cite{eng} and \cite{jech}. For example, if $\kappa$ is a cardinal number, then $\kappa^+$ denotes the first cardinal greater than $\kappa.$ 

\section{When  games favor Player I} 
  Let  $X$ be a topological space.   Denote by $\Tee$ the family of all non-empty open sets of $X$. For an ordinal number $\alpha$, let $\Tee^\alpha$ denotes the set of all sequences of the length $\alpha$ consisting of elements of $\Tee$.
  The space $X$ is called $\kappa$-\textit{favorable} whenever there exists    a  function
 $${\textit{\textbf s}} :\bigcup \{ (\Tee)^\alpha: \alpha<\kappa\} \to \Tee $$  such that
 for each sequence  $\{B_{\alpha+1}:\alpha<\kappa\}\subseteq\Tee$ with  $B_1
 \subseteq {\textit{\textbf s}}(\emptyset)$ and   $B_{\alpha+1} \subseteq {\textit{\textbf s}}(\left\{ B_{\gamma+1}:\gamma<\alpha\right\})$,
 for each $\alpha< \kappa$,    
  the union $\bigcup\{B_{\alpha+1}:\alpha<\kappa\}$  is dense in $X$. We may also say that the function ${\textit{\textbf s}}$ is witness to  $\kappa$-favorability of $X$. In fact, ${\textit{\textbf s}}$ is a winning strategy for Player I. For abbreviation we say that ${\textit{\textbf s}}$ is $\kappa$-winning strategy. Sometimes we do not precisely define a strategy. Just give hints  how a player should play. Note that, any winning strategy can be arbitrary on steps for  limit ordinals. 
  
  A family $\Bee$ of open non-empty subset is called a $\pi$\textit{-base} for $X$ if every non-empty open subset $U\subseteq X$ contains a member of $\Bee$. The smallest cardinal number $|\Bee|$, where $\Bee$ is a  $\pi$-base for $X,$ is denoted by $\pi(X).$

\begin{proposition}\label{prop1}  
Any topological space $X$ is $\pi(X)$-favorable.
\end{proposition}
\begin{proof}
Let $\{U_\alpha:\alpha<\pi(X)\}$ be a $\pi$-base. Put ${\textit{\textbf s}}(f)=U_\alpha$ for any sequence $f\in\Tee^\alpha.$ Each family  $\left\{B_\gamma: B_\gamma\subseteq U_\gamma\mbox{ and } \gamma<\pi(X)\right\}$ of open non-empty sets is again a $\pi$-base for $X.$ So, its union is dense in $X$.
\end{proof}

According to \cite[p. 86]{eng} the cellularity of $X$ is denoted by $\su(X)$. Let $\sat(X)$ be the smallest cardinal number $\kappa$ such that every  family of pairwise disjoint open sets of $X$ has cardinality 
$<\kappa,$ compare \cite{et}. 
Clearly, if $\sat(X)$ is a limit cardinal, then $\sat(X)=\su(X). $  In all other cases, $\sat (X) = \su (X)^+$. Hence,  $\su(X)\leq\sat(X)\leq\su(X)^+$.  
 Let  
$$\mu(X)=\min\{\kappa:X \text{ is a }\kappa\text{-favorable and $\kappa$ is a cardinal number}\}.$$

Proposition \ref{prop1} implies  $\mu(X)\leq\pi(X).$ 
The next proposition gives two natural strategies and gives more accurate estimation than 
$\su(X)\leq \mu(X)\leq\su^+(X)$

 \begin{proposition}\label{muu1} $\su(X)\leq \mu(X)\leq\sat(X)$.
\end{proposition}

\begin{proof}
Suppose $\su(X)>\mu(X) $. Fix a family $\{U_\xi:\xi<\mu(X)^+\}$ of pairwise disjoint open sets. If Player II always chooses  an  open set, which meets at most one  $U_\xi$, then he   will not lose the  open-open game of the length $ \mu(X)$, a contradiction.

Suppose sets $\{B_{\gamma+1}:\gamma<\alpha\}$ are chosen by Player II. If the set $$X\setminus\cl\bigcup\{B_{\gamma+1}:\gamma<\alpha\}$$ is non-empty, then Player I choses it.  Player I  wins  the open-open game of the length $\sat(X)$, when he will use this rule.  
 This gives  $\mu(X)\leq\sat(X).$
\end{proof}

Note that,  $\omega_0=\su(\{0,1\}^\kappa)=\mu(\{0,1\}^\kappa) \leq \sat(\{0,1\}^\kappa)=\omega_1$, where $\{0,1\}^\kappa$ is the  Cantor cube of weight $\kappa$. There exists a separable space $X$ which is not $\omega_0$-favorable, see A. Szyma\'nski \cite{szy} or \cite[p.207-208]{dkz}. Hence we get $$\omega_0=\su(X)<\mu(X)= \sat(X)=\omega_1.$$  

\section{On inverse systems with skeletal bonding maps}

Recall that,  a continuous surjection  is  \textit{skeletal} if for  any non-empty open sets $U\subseteq X$ 
the closure of  $f[U]$ has non-empty interior. If $X$ is a compact space and  $Y$ is a Hausdorff space, then a 
continuous surjection $f:X\to Y$ is skeletal if and only if $\int f[U] \not=\emptyset,$  for every non-empty 
and open $U\subseteq X$, see J. Mioduszewski and  L. Rudolf \cite{rm}.
 
\begin{lemma}\label{sk}  
 A  skeletal image of $\kappa$-favorable space is a $\kappa$-favorable space.
\end{lemma}
\begin{proof}
A proof follows by the same method as in  \cite[Theorem 4.1]{bjz}. In fact, repeat and generalize the proof given in \cite[Lemma 1]{kp9}.

\end{proof}

According to \cite{ch}, a directed set $\Sigma $ is said to be
\textit{$\kappa$-complete} if any chain of length $\leq \kappa$ consisting of its elements has the
least upper bound in $\Sigma$. 
 An inverse system $  \{ X_\sigma , \pi^\sigma_\varrho, \Sigma\}$ is said to be a $\kappa$-\textit{complete}, 
whenever  $\Sigma $ is 
$\kappa$-complete and for every chain $ A  \subseteq \Sigma$, where $|A|\leq \kappa$, such that 
$\sigma = \sup A \in \Sigma$ we get  $$X_{\sigma }= \varprojlim 
\{ X_{\alpha}, \pi^{\beta}_{\alpha}, A\}.$$ 
 In addition, we assume that  bonding maps are  surjections.

For $\omega$-favorability, the following lemma is  given without proof  in \cite[Corollary 1.4]{dkz}. We give a proof to convince  the reader that additional assumptions on topology  are unnecessary. 

\begin{lemma}\label{nn}
If $Y\subseteq X$  is  dense, then  $X$ is $\kappa$-favorable if and only if $Y$ is  $\kappa$-favorable. 
\end{lemma}   
\begin{proof}
Let a function $\sigma_X$ be a witness to $\kappa$-favorability of $X$. Put $$\sigma_Y(\emptyset)=\sigma_X(\emptyset)
\cap Y.$$ If  Player II chooses open set $V_1\cap Y\subseteq \sigma_Y(\emptyset)$, then put  $$V'_1=V_1\cap\sigma_X(\emptyset)
\subseteq \sigma_X(\emptyset).$$ We get  $V'_1\cap Y = V_1\cap Y\subseteq \sigma_Y(\emptyset)$,  since  $V_1\cap Y\subset\sigma_X(\emptyset) \cap Y$. 
Then we put $$\sigma_Y(V_1\cap Y)=\sigma_X(V'_1)\cap Y.$$  
Suppose  we have already defined  
$$\sigma_Y(\{V_{\alpha+1}\cap Y:\alpha<\gamma\})=\sigma_X(\{V'_{\alpha+1}:\alpha<\gamma\})\cap Y ,$$
for $\gamma<\beta<\kappa.$
If Player II chooses open set $V_{\beta+1}\cap Y\subseteq \sigma_Y(\{V_{\alpha+1}\cap Y:\alpha<\beta\})$, then put  $$V'_{\beta+1}=V_{\beta+1}\cap\sigma_X(\{V'_{\alpha+1}:\alpha<\beta\})
\subseteq \sigma_X(\{V'_{\alpha+1}:\alpha<\beta\}).$$ Finally, put 
$$ \sigma_Y(\{V_{\alpha+1}\cap Y:\alpha\leq\beta\})= \sigma_X(\{V'_{\alpha+1}:\alpha\leq\beta\}) \cap Y$$ and check that $\sigma_Y$ is witness to  $\kappa$-favorability of $Y$.

Assume that $\sigma_Y$ is a witness to $\kappa$-favorability of $Y$. If $\sigma_Y(\emptyset)=U_0\cap Y$ and $U_0\subseteq X$ is open, then  put 
$\sigma_X(\emptyset)=U_0.$ 
If Player II chooses open set $V_1\subseteq \sigma_X(\emptyset)$, then 
$V_1\cap Y\subseteq \sigma_Y(\emptyset).$   Put $\sigma_X(V_1)=U_1,$ where $\sigma_Y(V_1\cap Y)=U_1\cap Y$ and $U_1\subseteq X$ is open. 
Suppose  
$$\sigma_Y(\{V_{\alpha+1}\cap Y:\alpha<\gamma\})=U_\gamma\cap Y \mbox{  and } \sigma_X(\{V_{\alpha+1}:\alpha<\gamma\})=U_\gamma$$  have been already defined for $\gamma<\beta<\kappa.$   
If II Player chooses open set $V_{\beta+1}\subseteq \sigma_X(\{V_{\alpha+1}:\alpha<\beta\})$, then put  $\sigma_X(\{V_{\alpha+1}:\alpha<\beta+1\})=U_{\beta+1},$  where open set $U_{\beta+1}\subseteq X$ X is determined by $\sigma_Y(\{V_{\alpha+1}\cap Y:\alpha<\beta+1\})=U_{\beta+1}\cap Y.$ 
\end{proof}

The next Theorem is similar to \cite[Theorem 2]{bla}. We replace a continuous inverse system 
with  indexing set being a cardinal, by $\kappa$-complete inverse system, and also $\su(X)$ is  replaced  by  $\mu(X)$. Let $\kappa$ be a fixed cardinal number. 

\begin{theorem}\label{muu}
Let $X$ be a dense subset of  the inverse limit  of the $\kappa$-complete system $\{ X_\sigma , \pi^\sigma_\varrho, \Sigma\},$ where $\kappa=\sup\{\mu(X_\sigma):\sigma\in\Sigma \}.$ If all bonding maps are skeletal, then $\mu(X)=\kappa.$
\end{theorem}
\begin{proof} By Lemma \ref{nn}, one can  assume that $X=\invl\{ X_\sigma , \pi^\sigma_\varrho, \Sigma\}$.
 Fix  functions ${\textit{\textbf s}}_\sigma\colon\Tee_\sigma^{<\kappa}\to \Tee_\sigma$, each one is a  witness to  $\mu(X_\sigma)$-favorability of $X_\sigma$.  This does not reduce the generality, because $\mu(X_\sigma) \leq \kappa$ for every  $\sigma\in\Sigma$. In order to explain the induction, fix a bijection $f\colon \kappa\to\kappa\times \kappa$ such that: 
\begin{enumerate}
	\item if $f(\alpha)=(\beta,\zeta)$, then $\beta,\zeta\leq \alpha$; 
	\item $f^{-1}(\beta,\gamma)< f^{-1}(\beta,\zeta)$ if and only if $\gamma<\zeta$; \item $f^{-1}(\gamma,\beta)< f^{-1}(\zeta, \beta)$ if and only if $\gamma<\zeta$.   
\end{enumerate}

One can take as $f$ an isomorphism between $\kappa$ and $\kappa\times\kappa$, with canonical well-ordering,   see \cite{jech}. The function $f$ will indicate the strategy and sets that we have taken in the following induction.

We construct a function ${\textit{\textbf s}}\colon\Tee^{<\kappa}\to \Tee$ which will provide $\kappa$-favorability of $X$. The first step is defined for $f(0)=(0,0)$.
 Take  an arbitrary  $\sigma_1 \in \Sigma$ and put $${\textit{\textbf s}}(\emptyset)=\pi^{-1}_{\sigma_1}
({\textit{\textbf s}}_{\sigma_1}(\emptyset)).$$ Assume that  Player II chooses non-empty open set $B_1=\pi^{-1}_{\sigma_2}(V_1)\subseteq 
{\textit{\textbf s}}(\emptyset),$ where $V_1 \subseteq X_{\sigma_2}$ is open. 
Let 
$${\textit{\textbf s}}(\left\{ B_1\right\})=\pi^{-1}_{\sigma_1}({\textit{\textbf s}}_{\sigma_1}(\left\{ \int \cl \pi_{\sigma_1}(B_1)\cap{\textit{\textbf s}}_{\sigma_1}(\emptyset))\right\}))$$ and denote    $D^0_0=\int \cl \pi_{\sigma_1}(B_1)\cap {\textit{\textbf s}}_{\sigma_1}(\emptyset).$ So, after the first round and the next respond of Player I, we know: indexes $\sigma_1$ and $\sigma_2$, the open set $B_1\subseteq X$ and the open set $D^0_0\subseteq X_{\sigma_1}$.

 Suppose that  sequences of open sets $\left\{ B_{\alpha+1}\subseteq X:
\alpha<\gamma\right\}$, indexes $ \left\{ \sigma_{\alpha +1}:\alpha<\gamma\right\}$ and sets $\{D^\varphi_\zeta: f^{-1}(\varphi, \zeta)< \gamma \} $  have been already defined such that:   

 if $\alpha<\gamma$ and $f(\alpha)=(\varphi,\eta)$, then  $$B_{\alpha+1}=\pi^{-1}_{\sigma_{\alpha +2}}(V_{\alpha +1})\subseteq \textit{\textbf{s}}(\left\{ B_{\xi +1} :\xi<\alpha \right\})=\pi^{-1}_{\sigma_{\varphi+1}}({\textit{\textbf s}}_{\sigma_{\varphi+1}}(\left\{ D^\varphi_\nu:\nu<\eta\right\})),$$ 
where $D^\varphi_\nu=\int\cl \pi_{\sigma_{\varphi+1}}(B_{f^{-1}((\varphi,\nu))+1})\cap {\textit{\textbf s}}_{\sigma_{\varphi+1}}(\left\{ D^\varphi_\zeta:\zeta<\nu\right\})$ and  $V_{\alpha +1}\subseteq X_{\sigma_{\alpha +2}}$ are open.

If     $f(\gamma)=(\theta,\lambda)$ and  $\beta<\lambda$, then take  $$D^\theta_\beta=\int\cl \pi_{\sigma_{\theta+1}}(B_{f^{-1}((\theta,\beta))+1})\cap{\textit{\textbf s}}_{\sigma_{\theta+1}}(\left\{ D^\theta_\zeta:\zeta<\beta\right\})$$ 
and  put $$ {\textit{\textbf s}}(\left\{ B_{\alpha +1}:\alpha<\gamma \right\})=\pi^{-1}_{\sigma_{\theta+1}}({\textit{\textbf s}}_{\sigma_{\theta+1}}
(\left\{ D^\theta_\alpha:\alpha<\lambda\right\})).$$ Since $\Sigma$ is $\kappa$-complete, one  can assume that the  
sequence $\left\{ \sigma_{\alpha+1}:\alpha<\kappa\right\}$ is increasing and 
$\sigma=\sup\{\sigma_{\xi+1}:\xi<\kappa\}\in\Sigma$. 

We shall 
prove that $\bigcup_{\alpha<\kappa}B_{\alpha +1}$ is dense 
in $X$. Since $\pi^{-1}_\sigma(\pi_\sigma(B_{\alpha +1}))=B_{\alpha +1}$ for each $\alpha<\kappa$ and 
$\pi_\sigma$ is skeletal map, it is sufficient to
show that $\bigcup_{\alpha<\kappa}\pi_\sigma(B_{\alpha +1})$ is dense in $X_\sigma$. Fix arbitrary open 
set 
$(\pi^\sigma_{\sigma_{\xi+1}})^{-1}(W)$ where $W$ is an open set of $X_{\xi+1}$. Since ${\textit{\textbf s}}_{\sigma_{\xi+1}}$ is 
winning strategy on $X_{\sigma_{\xi+1}}$, there exists $D^\xi_\alpha$ such that $D^\xi_\alpha\cap W\neq 
\emptyset$, and $D^\xi_\alpha\subseteq\int \cl\pi_{\sigma_{\xi+1}}(B_{f^{-1}((\xi,\alpha))+1}).$ 
Therefore we get 
$$(\pi^\sigma_{\sigma_{\xi+1}})^{-1}(W)\cap \pi_\sigma(B_{\delta +1})\not =\emptyset,$$
where $\delta=f^{-1}((\xi,\alpha))$.
Indeed, suppose that $(\pi^\sigma_{\sigma_{\xi+1}})^{-1}(W)\cap \pi_\sigma(B_{\delta +1})=\emptyset.$ Then 
$$\emptyset=\pi^\sigma_{\sigma_{\xi+1}}[(\pi^\sigma_{\sigma_{\xi+1}})^{-1}(W)\cap 
\pi_\sigma(B_{\delta +1})]=W\cap\pi^\sigma_{\sigma_{\xi+1}}[\pi_{\sigma}(B_{\delta +1})]=W\cap\pi_{\sigma_{\xi+1}}(B_{\delta +1}).$$

Hence we have $W\cap\int\cl\pi_{\sigma_{\xi+1}}(B_{\delta +1} )=\emptyset,$ a contradiction.
\end{proof}

\begin{corollary}
If $X$ is dense subset of  an inverse limit of $\mu(X)$-complete system $\{ X_\sigma , \pi^\sigma_\varrho, \Sigma\},$ where all bonding map are skeletal, then
$$\su(X)=\sup\{\su(X_\sigma):\sigma\in\Sigma\}.$$
\end{corollary}

\begin{proof}
Let  $X=\invl\{ X_\sigma , \pi^\sigma_\varrho, \Sigma\}$. Since $\su(X) \geq \su( X_\sigma)$, for every $\sigma \in \Sigma$, we shall  show that
$$\su(X)\leq\sup\{\su(X_\sigma):\sigma\in\Sigma\}.$$  

Suppose that $\sup\{\su(X_\sigma):\sigma\in\Sigma\}=\tau<\su(X).$ 
Using  Proposition \ref{muu1} and Theorem \ref{muu}, check that
$$\mu(X)=\sup\{\mu(X_\sigma):\sigma\in\Sigma\}\leq\sup\{\su(X_\sigma)^+:\sigma\in\Sigma\}\leq \tau^+ \leq\su(X).$$ So, we get $\mu(X)=\su(X)=\tau^+$.  Therefore, there exists a family $\Raa$, of size $\tau^+$, which consists of pairwise disjoint open subset of $X$.  We can assume that $$\Raa\subseteq\{\pi^{-1}_\sigma(U):
U\text{ is an open subset of } X_\sigma \mbox{ and } \sigma\in\Sigma\}.$$ 
Since     $\{ X_\sigma , \pi^\sigma_\varrho , \Sigma\}$ is $\mu(X)$-complete inverse system and $|\Raa|=\mu(X)$, there exists $\beta\in\Sigma $ such that $$\Raa\subseteq\{\pi^{-1}_{\beta}(U):
U\text{ is an open subset of } X_{\beta}\},$$ a contradiction with $\su(X_\beta) < \tau^+$.
\end{proof}
 The above corollary is similar to \cite[Theorem 1]{bla}, but we replaced a continuous inverse system, whose indexing set is a cardinal number by $\kappa$-complete inverse system. 

\section{Classes  $\mathcal {C}_\kappa$}

Let $\kappa$ be an infinite  cardinal number.    Consider  inverse limits of  $\kappa$-complete system $  \{ X_\sigma , 
\pi^\sigma_\varrho, \Sigma\}$ with $\w(X_\sigma )\leq\kappa$. Let $\mathcal {C}_\kappa$ be a  class 
of   such inverse limits with  skeletal
bonding maps and $X_\sigma$ being  $T_0$-space.
Now, we show that the class $\mathcal{C}_\kappa$ is stable  under Cartesian  products.

\begin{theorem}\label{product}
The Cartesian product of spaces from $\mathcal{C}_\kappa$ belongs to  $\mathcal{C}_\kappa$.
\end{theorem}
\begin{proof}
Let $X=\prod\{X_s:s\in S\}$ where each $X_s \in \mathcal{C}_\kappa$. For each $s\in S$, let  $X_s=\varprojlim \{ X_{\sigma},s^{\sigma}_{\rho},\Sigma_s\}$ 
be a $\kappa$-complete inverse system with skeletal bonding map  such that each $T_0$-space  $X_\sigma$ has the weight  $\leq\kappa$. 
Consider the union  $$\Gamma=\bigcup\{\prod_{s\in A}\Sigma_s: A\in[S]^\kappa\}.$$ 
Introduce a partial order  on   $\Gamma$ as follows:  
 $$f\preceq g\Leftrightarrow \dom(f) \subseteq \dom (g)\text { and } \forall_{a\in\dom(f)} f(a)\leq_a 
g(a),$$ where $\leq_a$ is the partial order on $\Sigma_a$.  
 The set $\Gamma$  with the relation $\preceq$ is upward directed and  $\kappa$-complete. 
 
If $f\in \Gamma$, then $Y_f$ denotes   the Cartesian product $$\prod\{X_{f(a)}:a\in \dom(f)\}.$$ If $f\preceq g$, then put $$p^g_f=\left( \prod_{a\in\dom(f)} a^{g(a)}_{f(a)}\right) \circ\pi^{\dom(g)}_{\dom(f)},$$ where $\pi^{\dom(g)}_{\dom(f)}$ is the projection of $ \prod\{X_{g(a)}:a\in \dom(g)\}$ onto $ \prod\{X_{g(a)}:a\in \dom(f)\}$ and  $\prod_{a\in\dom(f)} a^{g(a)}_{f(a)}$ is the Cartesian product of the bonding maps $a^{g(a)}_{f(a)}: X_{g(a)}\to X_{f(a)}$. We get  the 
inverse system$\{Y_f,p^g_f,\Gamma\}$ which is $\kappa$-complete,  bonding maps are skeletal and 
$\w(Y_f)\leq\kappa$. So, we can take  $Y=\varprojlim\{Y_f,p^g_f,\Gamma\}$. 

Now, define a map $h: X\to Y$ by the  formula: 
$$h(\{x_s\}_{s\in S})=\{x_f\}_{f\in\Gamma},$$ 
where $x_f=\{x_{f(a)}\}_{a\in\dom(f)}\in Y_f$ and $f\in\prod\{\Sigma_a:a\in\dom(f)\}$ and $\dom(f)\in[S]^\kappa.$ 
By the property $$\{x_s\}_{s\in S}=\{t_s\}_{s\in S}\Leftrightarrow\forall_{s\in S}\forall_{\sigma\in\Sigma_s}\; x_\sigma=t_\sigma\Leftrightarrow\forall_{f\in\Gamma}\; x_f=t_f,$$
 the map $h$ is well defined and  it is injection. 

The map $h$ is surjection. Indeed, let $\{b_f\}_{f\in\Gamma}\in Y.$
For each  ${s\in S}$ and each ${\sigma\in\Sigma_s}$ we fix  $f^s_\sigma\in \Gamma$ such that $ s\in\dom(f^s_\sigma) $  
and 
$ f^s_\sigma(s)=\sigma.$
Let $\pi_{f(s)}:Y_f\to X_{f(s)}$ be a projection for each $f\in\Gamma.$

For each $t\in S$ let define $b_t=\{b_\sigma\}_{\sigma\in\Sigma_t},$ where $b_\sigma=\pi_{f^t_\sigma(t)}(b_{f^t_\sigma}).$  We shall prove that an element $b_t$ is a 
thread of the space $X_t$. 
Indeed, if $\sigma\geq \rho$ and $\sigma,\rho\in\Sigma_t$, then take functions $f^t_\sigma$ and $g^t_\rho$. For abbreviation, denote $f=f^t_\sigma$ and $g=g^t_\rho.$ Define a 
function $h\colon \dom(f)\cup \dom(g)\to \bigcup\{\Sigma_t:t \in\dom(f)\cup \dom(g)\} $ in the following 
way:
\[h(s)=
\begin{cases}
g(s) &\text{ if }s\in\dom(g)\setminus\dom(f)\\
f(s) &\text{ if }s\in\dom(f).
\end{cases}\]
 The function $h$ is element of $\Gamma$ 
and 
$f,g\preceq h$. Note that $h|\dom(f)=f$ and $h|\dom(g)\setminus \{t\}=g|\dom(g)\setminus \{t\}.$
Since

\begin{align*} 
\left\{b_{g(s)}\right\}_{s\in \dom(g)}=b_{g}=p^h_{g}(b_h)=\left(\prod_{s\in\dom(g)} s^{h(s)}_{g(s)}\right)\left(\pi^{\dom(h)}_{\dom(g)}
(b_h)\right)=\\
\left(\prod_{s\in\dom(g)} s^{h(s)}_{g(s)}\right)
\left(\{b_{h(s)}\}_{s\in \dom(g)}\right)= 
\left\{s^{h(s)}_{g(s)}(b_{h(s)})\right\}_{s\in \dom(g)}\end{align*} 
we get
$$b_\rho=b_{g(t)}=s^{h(t)}_{g(t)}(b_{h(t)})=s^{f(t)}_{g(t)}(b_{f(t)})=s^\sigma_\rho(b_{\sigma}).$$
 It is clear that $h(\{a_t\}_{t\in S})=\{b_f\}_{f\in\Gamma}.$

We shall prove that the map $h$ is continuous. Take an open subset  $U=\prod_{s\in\dom(f)} A_{f(s)}\subseteq Y_f$  such that
\[
A_{f(s)}=
\begin{cases}
V,&\text{ if } s=s_0;\\
X_{f(s)},&\text{ otherwise, }
\end{cases}\]
where $V\subseteq X_{f(s_0)}$ is open subset. A map $p_f$ is projection from the inverse limit $Y$ to $Y_f$. It is sufficient to show that :
$$h^{-1}((p_f)^{-1}(U))=\prod_{s\in S}B_s$$
where 
\[
B_s=
\begin{cases}
W,&\text{ if } s=s_0;\\
X_s,&\text{ otherwise, }
\end{cases}\]
and $W=\pi_{f(s_0)}^{-1}(V)$ and $\pi_{f(s_0)}\colon Y_{f}\to X_{\sigma_0}$ is the projection 
and $f(s_0)=\sigma_0$.
We have 
\begin{align*}\{x_s\}_{s\in S}\in h^{-1}((p_f)^{-1}(U))\Leftrightarrow  p_f(h(\{x_s\}_{s\in S}))\in U \\ \Leftrightarrow p_f(\{x_f\}_{f\in 
\Gamma})=x_f\in U\Leftrightarrow x_{f(s_0)}\in V \\
\Leftrightarrow x_{s_0}\in W\Leftrightarrow x\in 
\prod_{s\in S}B_s\subseteq \prod_{s\in S}X_s=X\end{align*}

Since the map $h$ is bijection and  $$(p_f)^{-1}(U))=h(h^{-1}((p_f)^{-1}(U)))=h(\prod_{s\in S}B_s)$$ for 
any subbase subset $\prod_{s\in S}B_s\subseteq X$, the map $h$ is open. 
\end{proof}

In the case $\kappa=\omega$ we have well know results that I-favorable space is productable (see \cite{dkz} or \cite{kp7}) 

\begin{corollary}\label{prod}
Every I-favorable space is  stable under any product.
\end{corollary} 

If $D$ is a set and $\kappa$ is cardinal number then   we denote  $\bigcup_{\alpha<\kappa} D^\alpha$ by $D^{<\kappa}$.

The following result probably is known but we give a proof for the sake of completnes.

\begin{theorem}\label{mu2}
Let $\kappa$ be an infinite cardinal and let $T$ be a set such that $|T|\geq \kappa^\kappa$.  If $A\in[T]^\kappa$ and $f_\delta\colon T^{<\kappa}\to T$  for all $\delta<\kappa^{<\kappa}$ then there exists a set $B\subseteq T$ such that $|B|\leq \tau$ and $A\subseteq B$ and $f_\delta(C)\in B$ for every $C\in B^{<\kappa}$ and every $\delta<\kappa^{<\kappa}$, where  $$\tau=
\begin{cases}
\kappa^{<\kappa} &\text{ for regular }\kappa\\
\kappa^\kappa & \text{ otherwise.} 
\end{cases}$$ 
\end{theorem}

\begin{proof}
Assume that $\kappa$ is regular cardinal.
Let $A\in[T]^\kappa$ and let $f_\delta\colon\bigcup_{\alpha<\kappa} T^\alpha\to T$ for $\delta<\kappa^{<\kappa}.$
Let $A_0=A$. Assume that we have defined $A_\alpha$ for $\alpha<\beta$ such that $|A_\alpha|\leq \kappa^{|\alpha|}$. Put
$$A_\beta=(\bigcup_{\alpha<\beta} A_\alpha)\cup\{f_\delta(C):C\in (\bigcup_{\alpha<\beta} A_\alpha)^{< \beta}\text{ and }\delta<\kappa^{|\beta|}\}.$$ 
Calculate the size of the set $A_\beta$:

$$|A_\beta|\leq |(\bigcup_{\alpha<\beta} A_\alpha)||\kappa^{|\beta|}||(\bigcup_{\alpha<\beta} A_\alpha)^{
<\beta}|\leq \kappa^{|\beta|}|(\kappa^{|\beta|})^{|\beta|}|
\leq \kappa^{|\beta|}.$$

Let $B=\bigcup_{\beta<\kappa} A_\beta$, so we get $|B|\leq\kappa^{<\kappa}$. 
 Fix  a sequence $\left\langle b_\alpha:\alpha<\beta\right\rangle\subseteq B$  and $f_\gamma$. Since $\cf(\kappa)=\kappa$  there exists $\delta<\kappa$ such that $C= \{b_\alpha:\alpha<\beta\}\subseteq A_{\delta}$ and $f_\gamma(C)\in A_{\sigma+1}$ for some $\sigma<\kappa$.

In the second case  $\cf(\kappa)<\kappa$,  we proceed the above induction up to $\beta=\kappa$. Let
$B=A_\kappa$, so  we get $|B|\leq\kappa^\kappa$ and $B=\bigcup_{\beta<\kappa^+} A_\beta$. 
Similarly to the first case we get that $B$ is closed under all function $f_\delta$, $\delta<\kappa^{<\kappa}$.

\end{proof}

\begin{theorem}\label{mu1}
If $X$ belongs to the class $\mathcal{C}_\kappa$ then $\su(X)\leq\kappa$ .
\end{theorem}
\begin{proof}
If $X\in\mathcal{C}_\kappa$ then by Theorems \ref{muu} and \ref{muu1} we get  $\su(X)\leq\mu(X)\leq\kappa$.
\end{proof}

We apply some facts from the paper \cite{kp8}. Let $\Pee$ be a family of open subset of topological space $X$ and $x,y\in X$. We say that $x\sim_{\Pee} y$ if and only if $x\in V\Leftrightarrow y\in V$ for every $V\in\Pee.$ The family of all sets $[x]_{\Pee}=\{y:y\sim_{\Pee}x\}$ we denote by $X/\Pee$. Define a map 
$q:X\to X/\Pee$ as follows $q[x]=[x]_{\Pee}$. The set $X/\Pee$ is equipped with topology $\Tee_{\Pee}$ generated by 
all images $q[V]$ where $V\in\Pee$. 

Recall Lemma 1 from paper \cite{kp8}:
\textit{
If $\Pee$ is a family of open set of $X$ and $\Pee$ is closed under finite intersection then 
the mapping $q:X\to X/\Pee$ is continuous. Moreover if $X=\bigcup \Pee$ then the family 
$\{q[V]: V\in \Pee\}$ is a base for the topology $\Tee_{\Pee}$.}

Notice that if $\Pee$ has a property
\begin{align*}
 \tag{seq}\forall(W\in\Pee)\exists(\{V_n:n<\omega\}\subseteq\Pee)\exists(\{U_n:n<\omega\}\subseteq\Pee)\\
W=\bigcup_{n<\omega}\;U_n \text{ and }\forall(n<\omega)\;U_n\subseteq X\setminus V_n\subseteq U_{n+1}
\end{align*}
then $\bigcup\Pee=X$ and by \cite[Lemma 3]{kp8} the topology $\Tee_{\Pee}$ is Hausdorff. Moreover if
$\Pee$ is closed under finite intersection then by \cite[Lemma 4]{kp8} the topology $\Tee_{\Pee}$ is regular.
Theorem 5  and Lemma 9 \cite{kp8} yeield.

\begin{theorem}\label{closed}
If $\Pee$ is a set of open subset of topological space $X$ such that:
\begin{enumerate}
	\item is closed under $\kappa$-winning strategy, finite union and finite intersection,
	\item has property (seq),
\end{enumerate}
then $X/\Pee$ with topology $\Tee_{\Pee}$ is completely regular space and $q:X\to X/\Pee$ is skeletal. \hfill $\Box$
\end{theorem}
 
 If a topological space $X$ has the cardinal number $\mu(X)=\omega$ then $X\in \Cee_\omega$, but 
for $\mu(X)$ equals for instance $\omega_1$ we do not even know if $X\in\Cee_{{\omega_1}^\omega}.$ 
 
\begin{theorem}\label{mu3}Each Tichonov  space $X$  with $\mu(X)=\kappa$ can be dense embedded into  inverse limit of a system $\{ X_\sigma , \pi^\sigma_\varrho, \Sigma\},$ where all bonding map are skeletal,
indexing set $\Sigma$ is $\tau$-complete each $X_\sigma$ is Tichonov space with $\w(X_\sigma)\leq\tau$     and
$$\tau=
\begin{cases}
\kappa^{<\kappa} &\text{ for regular }\kappa\\
\kappa^\kappa & \text{ otherwise. }
\end{cases}$$
\end{theorem}
\begin{proof}
Let $\Bee$ be a $\pi$-base for topological space $X$ consisting of cozerosets and $\sigma :\bigcup \{ \Bee^\alpha: \alpha<\kappa\} \to \Bee $ be a $\kappa$-winning strategy. We can define a function of finite intersection property and finite union property as following :
$g(\left\{ B_0,B_1,\ldots,B_n\right\})=B_0\cap B_1\cap\ldots\cap B_n$ and 
$h(\left\{ B_0,B_1,\ldots,B_n\right\})=B_0\cup B_1\cup\ldots\cup B_n$.
For each cozeroset $V\in\Bee$ fix a continuous function $f_V: X \to [0,1]$ such that $V= f_V^{-1}((0,1])$. Put  $\sigma_{2n} (V) = f_V^{-1}((\frac 1 n,1])$ and $\sigma_{2n+1} (V) = f_V^{-1}([0, \frac 1 n))$.
 By Theorem \ref{mu2} for each $\Raa\in[\Bee]^{\kappa}$  and all functions  $h, g, \sigma_n, \sigma$ there is
subset $\Pee\subseteq\Bee$ such that:
\begin{enumerate}
	\item $|\Pee|\leq\tau$ where 
	$$\tau=
\begin{cases}
\kappa^{<\kappa} &\text{ for regular }\kappa\\
\kappa^\kappa & \text{ otherwise, }
\end{cases}$$
\item $\Raa\subseteq\Pee$,
\item $\Pee$ is closed under $\kappa$-winning strategy $\sigma$, function of finite intersection property and finite union property ,
\item $\Pee$ is closed under $\sigma_n$ , $n<\omega$, hence $\Pee$ holds property (seq).
\end{enumerate}

Therefore by Theorem \ref{closed} we get skeletal mapping $q_{\Pee}:X\to X/\Pee .$ 
Let $\Sigma\subseteq [\Bee]^{\leq\tau}$ be a set of families which satisfies above condition
$(1), (2), (3),(4)$. If $\Sigma $ is directed by inclusion. It is easy to check that $\Sigma$ is $\tau$-complete. 
Similar to  \cite[Theorem 11]{kp8} we define a function $f:X\to Y$ as following $f(x) = \{f_\Pee(x)\},$
where $f(x)_\Pee=q_\Pee(x)$ and $ Y= \varprojlim \{ X/\Raa, q^\Raa_\Pee,\Cee\}$.
 If  $\Raa, \Pee \in \Cee$ and  $\Pee \subseteq \Raa$, then  $q^\Raa_\Pee (f(x)_\Raa)=f(x)_\Pee$. Thus  $f(x)$ is a thread, i.e. $f(x)\in Y$. It easy to see that $f$ is homeomorphism onto its image and $f[X]$ is dense in $Y$, compare \cite[proof of Theorem 11]{kp8}  
 
\end{proof}

 The Theorem \ref{mu3} suggests a question: 
 
 \textit{Does each space $X$ belong to $\Cee_{\mu(X)}$? }
 
  Fleissner \cite{fl} proved that there exists a space $Y$ such that $\su(Y)=\aleph_{0}$ and $\su(Y^3)=\aleph_{2}$. Hence  we get $\mu(Y)=\aleph_1$, by Theorem \ref{muu} and Corollary \ref{prod}. Suppose that $Y\in\Cee_{\mu(X)}$ then $\su(Y^3)\leq\aleph_1$, by Theorem \ref{mu1}, a contradiction.

\begin{corollary}If $X$ is topological space with $\mu(X)=\kappa$  then  
$\su(X^I)\leq\tau$  and
$$\tau=
\begin{cases}
\kappa^{<\kappa} &\text{ for regular }\kappa\\
\kappa^\kappa & \text{ otherwise. }
\end{cases}$$
\end{corollary}
\begin{proof}
By Theorem \ref{mu2} we get $X^I\in\mathcal{C}_\tau$. Hence by  Theorem \ref{mu1} and  \ref{product} we have $\su(X^I)\leq\tau$.
\end {proof}

By above Corollary we get the following

\begin{corollary}\cite[Kurepa]{ku62}
If $\{X_s:s\in S\}$ is a family of topological spaces and $\su(X_s)\leq\kappa$ for each $s\in S$, 
then $\su(\prod\{X_s: s\in S\})\leq 2^\kappa.$ \hfill $\Box$
\end{corollary}


\end{document}